\documentclass{amsart}

\usepackage{amssymb,amsthm,amstext,amsfonts,amsmath,yfonts,latexsym}

\newcommand{\inte}[2][X]{\mathrm{int}_{#1}(#2)}
\newcommand{\cl}[2][X]{\mathrm{cl}_{#1}(#2)}
\newcommand{\bd}[2][X]{\mathrm{bd}_{#1}(#2)}
\newcommand{\pair}[1]{{\langle #1 \rangle}}
\renewcommand{\c}{\mathfrak{c}}
\newcommand{\B}{\mathcal{B}}
\newcommand{\V}{\mathcal{V}}
\newcommand{\D}{\mathcal{D}}
\newcommand{\G}{\mathcal{G}}
\newcommand{\covM}{\mathrm{cov}(\mathcal{M})}
\newcommand{\cofM}{\mathrm{cof}(\mathcal{M})}
\newcommand{\Pint}{\mathfrak{p}}
\newcommand{\class}{\mathbf{K}}
\newcommand{\poset}{\mathbb{P}}

\renewcommand{\d}{\mathfrak{d}}
\newcommand{\baire}{{}\sp\omega\omega}
\newcommand{\en}{\mathcal{N}}
\newcommand{\ctble}{\textbf{ctble}}
\newcommand{\sigmacent}{\sigma\textbf{-cent}}
\newcommand{\MA}[2][\kappa]{\mathbf{MA}_{#2}(#1)}
\newcommand{\PFA}{\mathbf{PFA}}
\newcommand{\remcard}{\mathbf{\chi_R}}
\newcommand{\farcard}{\mathbf{\chi_F}}
\newcommand{\Q}{\mathbb{Q}}
\newcommand{\cantor}{{}\sp{\omega}2}
\newcommand{\cantortree}{{}\sp{<\omega}2}
\newcommand{\model}{\mathbf{M}}
\newcommand{\random}{\mathbb{B}}

\newtheoremstyle{theorem}
     {11pt}
     {11pt}
     {}
     {}
     {\bfseries}
     {}
     {.5em}
     {\noindent\thmnumber{#2}. \thmname{#1}\thmnote{#3}}

\theoremstyle{theorem}

\newtheorem{ques}{Question}[section]
\newtheorem{lemma}[ques]{Lemma}
\newtheorem{propo}[ques]{Proposition}
\newtheorem{thm}[ques]{Theorem}
\newtheorem{coro}[ques]{Corollary}

\title{Far points and discretely generated spaces}
\author[Dow]{Alan Dow}
\author[Hern\'andez-Guti\'errez]{Rodrigo Hern\'andez-Guti\'errez}
\address{Department of Mathematics and Statistics, University of North Carolina at Charlotte, Charlotte NC 28223}
\email[Dow]{adow@uncc.edu}
\email[Hern\'andez-Guti\'errez]{rodrigo.hdz@gmail.com}

\date{\today}

\begin{document}

\maketitle

\begin{abstract}
We give a partial solution to a question by Alas, Junqueria and Wilson by proving that under $\PFA$ the one-point compactification of a locally compact, discretely generated and countably tight space is also discretely generated. After this, we study the cardinal number given by the smallest possible character of remote and far sets of separable metrizable spaces. Finally, we prove that in some cases a countable space has far points.
\end{abstract}

\section*{}

All spaces under discussion will be Tychonoff.\vskip6pt

Let $X$ be any space. We say that $X$ is \emph{discretely generated} if for every $A\subset X$ and $p\in\cl{A}$ there is a discrete subset $D\subset A$ such that $p\in\cl{D}$. Also, $X$ is \emph{weakly discretely generated} if every non-closed subset $A$ in a space $X$ contains a discrete subset $D\subset A$ such that $\cl{D}\setminus A\neq\emptyset$. These notions were defined in \cite{disc_gen_first} and have been a subject of study lately. First countable spaces are clearly discretely generated and in fact spaces with relatively rich structures are discretely generated. For example, it has been shown that countable products and box products of monotonically normal spaces are discretely generated (\cite{disc_gen_box}, \cite{alas_junq_wil_class_disc_gen}). Also, notice that the definitions of discretely generated and weakly discretely generated spaces are similar to the classic notions of Frech\'et-Urysohn spaces and sequential spaces, respectively, with the advantage that there is no bound for tightness (see Example 3.5 and Proposition 3.6 in \cite{disc_gen_first}).

If one focuses on compact spaces, there are interesting things to say. Although there are compact spaces that are not discretely generated \cite[Example 4.3]{disc_gen_first}, all compact spaces are weakly discretely generated \cite[Proposition 4.1]{disc_gen_first}. Also, a dyadic compactum is discretely generated if and only if it is metrizable \cite[Theorem 2.1]{disc_gen_box}. In this paper, we focus on the following question by Alas, Junqueira and Wilson.

\begin{ques}\cite{alas_junq_wil_class_disc_gen}\label{theques}
Let $X$ be a locally compact and discretely generated space. Is the one-point compactification of $X$ also discretely generated?
\end{ques}

So far, this question has been answered consistently in the negative. First countable counterexamples have been constructed using CH \cite{alas_junq_wil_class_disc_gen} or the existence of a Souslin line \cite{aurichi}. Later, it was shown in \cite{hg-remotediscretelygen} that a similar construction could be carried out assuming the cardinal equation $\mathfrak{p}={cof}(\mathcal{M})$; the resulting space has an open dense set homeomorphic to $\omega\times{}\sp{\omega}2$ and the rest of the space is homeomorphic to the ordinal space $\mathfrak{p}$.

It is the objective of this paper to show that it is consistent that Question \ref{theques} has a partial affirmative answer. 

\begin{thm}\label{main}
Assume $\PFA$. Then every locally compact discretely generated space of countable tightness has its one-point compactification discretely generated.
\end{thm}

So, under $\PFA$, there are no first countable counterexamples to Question \ref{theques} but certainly the example from \cite{hg-remotediscretelygen} does exist. Thus, Question \ref{theques} is still open.

It is interesting to notice that all counterexamples to Question \ref{theques} mentioned above are constructed using the following procedure. Consider some simple space $X$ and find a point $p\in\beta X\setminus X$ that is remote (that is, $p$ is not in the closure of any nowhere dense subset of $X$) and has some additional property. Then some appropriate quotient of $\beta X\setminus\{p\}$ will be a counterexample to Question \ref{theques}. The counterexample from \cite{hg-remotediscretelygen} used a remote point of $\omega\times{}\sp{\omega}{2}$ with characer $\Pint$ (which, consistently, may not exist). For this reason, the second-named author asked in \cite{hg-remotediscretelygen} whether the minimal character of a remote set of a separable metrizable space is equal to any well-known small cardinal. In section \ref{sectionchar} of this paper we improve the previous existing bound given in \cite{hg-remotediscretelygen}, see Theorems \ref{notcofmeager} and \ref{farcardinals}. Although we were not able to describe this cardinal completely (Question \ref{remotecardquestion}), we did calculate the minimal character of a \emph{far} set of a separable metrizable space, see Corollary \ref{farchar}.

For our final section, let us comment a little on the history of remote points. The existence of remote points of separable metrizable spaces was a hard problem that was finally solved by Chae and Smith \cite{chae-smith} and, independently, by van Douwen \cite{vd51}. A restriction finally proved in \cite{dow-sep_no_remote} is that pseudocompact spaces have no remote points. For some values of $\kappa$, ccc non-pseudocompact spaces of $\pi$-weight $\kappa$ have been shown to have remote points: this was proved in \cite{chae-smith} and \cite{vd51} for $\kappa=\omega$,  in \cite{dow-remote_omega1} for $\kappa=\omega_1$ and in \cite{brown_dow}, assuming CH, for $\kappa=\omega_2$. Also, the spaces that have been constructed with no remote points (for example, see \cite{vd-vm-remote}, \cite{dow-sep_no_remote}, \cite{dow-peters-products_remote}) are not ccc. So the following question seems natural.

\begin{ques}\cite[Question 1]{dowquestions}
Does every ccc non-pseudocompact space have a remote point?
\end{ques}

We also have to consider that it is consistent that there is a separable space with no remote points \cite{dow-sep_no_remote}. For this reason, we find it is interesting to study remote and far points in countable and crowded spaces. In this paper we focus on the following question.

\begin{ques}
Does every countable and crowded space have a far point?
\end{ques}

This question is answered under some additional conditions: if the character and $\pi$-character coincide at every point (Theorem \ref{farpicharischar}), when the $\pi$-weight of the space has uncountable cofinality (Corollary \ref{farregular}) and when we assume that $\c\leq\aleph_{\omega+1}$ (Corollary \ref{farforsmall}). However, this problem remains open in its full generality.

\section{Preliminaries}

In a topological space $X$, given $A\subset X$, the closure of $A$ will be denoted by $\cl{A}$ and its boundary by $\bd{A}$. A standard reference for topological concepts is of course \cite{eng}.

The Cantor set is the set $\cantor$ of all functions $\omega\to2$. The Cantor tree is $\cantortree=\bigcup\{{}\sp{n}2:n<\omega\}$. Given $s\in\cantortree$, let $[s]=\{f\in\cantor:s\subset f\}$. The set $\{[s]:s\in\cantortree\}$ is a base for the topology of $\cantor$.

Recall that in any compact Hausdorff space, the character and the pseudocharacter of closed subsets are equal. The character of a point $p$ in a space $X$ will be denoted by $\chi(p,X)$.

For every Tychonoff space $X$, $\beta X$ denotes the \v Cech-Stone compactification of $X$ and $X\sp\ast=\beta X\setminus X$. A point in the \v Cech-Stone remainder $p\in X\sp\ast$ is called \emph{remote} (\emph{far}) if very time $N\subset X$ is a nowhere dense set (a closed discrete set, respectively) then $p\notin\cl[\beta X]{N}$. Extending the definition of remote point, we will say that $F\subset\beta X$ is a \emph{remote set} (\emph{far set}) of $X$ if $F\cap\cl[\beta X]{N}=\emptyset$ for every nowhere dense subset (closed discrete subset, respectively) $N$ of $X$.

In this paper we will use some some small uncountable cardinals. See \cite{vd62} or \cite{vaughan-small} for topological introductions, \cite{blass} for a recent survey in a set-theoretic perspective and \cite{bart} for information on consistency results. Basically, we will use $\Pint$, $\covM$, $\cofM$ and $\d$.

Let $\baire$ be the family of all functions from $\omega$ to $\omega$. The cardinal $\d$ is the size of the smallest family $\mathfrak{G}\subset\baire$ such that every time $f:\omega\to\omega$, there is $g\in\mathfrak{G}$ such that $\{n<\omega:g(n)<f(n)\}$ is finite. 

We will assume the terminology of posets such as dense subsets and filters and so on, see \cite[Section 1.4]{bart}. If $\class$ if a class of partially ordered sets and $\kappa$ is a cardinal number, $\MA[\class]{\kappa}$ is the following assertion:
\begin{quote}
If $\poset\in\class$ and $\{D_\alpha:\alpha<\kappa\}$ are dense subsets of $\poset$, there is a filter $G\subset\poset$ that intersects $D_\alpha$ for every $\alpha<\kappa$.
\end{quote}

Let $\ctble$ be the class of countable posets and $\sigmacent$ the class of $\sigma$-centered posets.

\begin{thm}\label{martins}
\begin{itemize}
\item[(a)] $\kappa<\Pint$ if and only if $\MA[\sigmacent]{\kappa}$. (\cite[14C, p. 25]{fremlin}) 
\item[(b)] $\kappa<\covM$ if and only if $\MA[\ctble]{\kappa}$. (\cite[Theorem 2.4.5]{bart})
\end{itemize}
\end{thm}

We will assume some knowledge of forcing and random reals for the proof of Lemma \ref{lemmarandom} and Theorem \ref{notcofmeager} and of course we will assume knowledge about $\PFA$ in the proof of the main Theorem \ref{main}.

\section{Discretely generated spaces}

In this section we will prove Theorem \ref{main}. We will split the proof in several steps to make it more transparent.

Recall that a free $\omega_1$-sequence in a space $X$ is an ordered subset $\{x_\alpha:\alpha<\omega_1\}$ such that for all $\beta<\omega_1$ we have $\cl{\{x_\alpha:\alpha<\beta\}}\cap \cl{\{x_\alpha:\beta\leq\alpha<\omega_1\}}=\emptyset$.

\begin{lemma}\label{freeseq}
Let $K$ be a compact space and $p\in K$ such that $K\setminus\{p\}$ is countably tight, $p$ is not isolated and $p$ is not in the closure of any countable discrete subset of $K$. Then there is a free $\omega_1$-sequence in $K$ such that $p$ is its only complete accumulation point.
\end{lemma}
\begin{proof}
We will recursively construct our free $\omega_1$-sequence $\{x_\alpha:\alpha<\omega_1\}$. We will also define a sequence of open neighborhoods $\{U_\alpha:\alpha<\omega_1\}$ of $p$ such that $\{x_\alpha:\beta\leq\alpha<\omega_1\}\subset U_\beta$ and $\{x_\alpha:\alpha<\beta\}\cap\cl[K]{U_\beta}=\emptyset$ for every $\beta<\omega_1$. Notice that this implies that the sequence constructed in this way is a free $\omega_1$-sequence.

In step $\beta<\omega_1$, assume that we have constructed $\{x_\alpha:\alpha<\beta\}$. Since $\{x_\alpha:\alpha<\beta\}$ is countable and discrete, $p$ cannot be in its closure. Thus, there is an open set $U_\beta$ with $p\in U_\beta$ and $\cl[K]{U_\beta}\cap\cl[K]{\{x_\alpha:\alpha<\beta\}}=\emptyset$. Notice that the pseudocharacter at $p$ is uncountable. Thus, choose a point $x_\beta\in\bigcap\{U_\alpha:\alpha\leq\beta\}$ such that $x_\beta\neq p$. This completes the construction.

Now, notice that by countable tightness, no point of $K\setminus\{p\}$ can be a complete accumulation point of $\{x_\alpha:\alpha<\beta\}$. However, by compactness this sequence must have a complete accumulation point, it must then be $p$.
\end{proof}

\begin{lemma}\label{needfreeseq}
Let $Y$ be a space and $p\in Y$ such that $Y\setminus\{p\}$ is countably tight and discretely generated. Assume that $A$ is dense in $Y$ and $p\notin A$. If $p$ is in the closure of a countable discrete subset of $Y$ then $p$ is also in the closure of some countable discrete subset of $A$.
\end{lemma}
\begin{proof}
Let $\{p_i:i<\omega\}\subset Y$ be a discrete subset that clusters at $p$. For each $i<\omega$, let $U_i$ be an open set with $p_i\in U_i$ such that $U_i\cap U_j=\emptyset$ every time $i\neq j<\omega$. Then find a sequence $S_n\subset A\cap U_i$ that clusters at $p_i$ for each $i<\omega$. Then $\bigcup\{S_i:i<\omega\}$ is a discrete subset of $A$ that clusters at $p$.
\end{proof} 

\begin{lemma}\label{sequentially}
Let $K$ be a compact space, $T\subset K$ closed with character $<\Pint$ and $A\subset K$ such that $A$ is countable and $\cl[K]{A}\cap T\neq\emptyset$. Then there exists $B\subset A$ such that for every open set $U$ with $T\subset U$, the set $B\setminus U$ is finite. 
\end{lemma}
\begin{proof}
Enumerate $A=\{x_n:n<\omega\}$ faithfully. Let $\{U_\alpha:\alpha<\kappa\}$ be a family of open sets such that $\kappa<\Pint$ and $T=\bigcap\{U_\alpha:\alpha<\kappa\}$. For each $\alpha<\kappa$, let $A_\alpha=\{n<\omega:x_n\in U_\alpha\}$.

Now, consider the poset $\poset$ of all $p=\pair{s_p,F_p}$, where
\begin{itemize}
\item[(a)] $s_p\in[\omega]\sp{<\omega}$, and
\item[(b)] $F_p\in[\kappa]\sp{<\omega}$.
\end{itemize}
and $q\leq p$ if
\begin{itemize}
\item[(c)] $s_p\subset s_q$,
\item[(d)] $F_p\subset F_q$, and
\item[(e)] for all $n\in s_q\setminus s_p$ and $\alpha\in F_p$, $n\in A_\alpha$
\end{itemize}

Let us prove that $\pair{\poset,\leq}$ is $\sigma$-centered. It is enough to show that given $s\in[\omega]\sp{<\omega}$, the set $\poset_s=\{p\in\poset:s_p=s\}$ is centered. Given $p,q\in\poset_s$, it is easy to see that $\pair{s,F_p\cup F_q}$ is an element of $\poset_s$ that extends both $p$ and $q$.

For each $n<\omega$, consider $D_n=\{p\in\poset:s_p\setminus n\neq\emptyset\}$ and for each $\alpha\in\kappa$, consider $E_\alpha=\{p\in\poset:\alpha\in F_p\}$. Let us argue why these sets are dense in $\poset$, let $p\in\poset$. Given $n<\omega$, since $\bigcap\{U_\alpha:\alpha\in F_p\}$ is a neighbordood of $T$, it is infinite so there is $r<\omega$ such that $x_r\in\bigcap\{U_\alpha:\alpha\in F_p\}$ and $r\geq n$. Then $\pair{s_p\cup\{r\},F_p}$ is an element of $D_n$ that extends $p$. Given $\alpha\in\kappa$, $\pair{s_p,F\cup\{\alpha\}}$ is an element of $E_\alpha$ that extends $p$.

Now, since $\kappa<\Pint$, there is a filter $G\subset\poset$ such that $G\cap D_n\neq\emptyset$ for each $n<\omega$ and $G\cap E_F$ for each $F\in[\kappa]\sp{<\omega}$. Let $B=\{x_n:n\in s_p\textrm{ for some }p\in G\}$. Since $G\cap D_n\neq\emptyset$ for each $n<\omega$, we get that $|B|=\omega$.

We claim that $B$ is the subset of $A$ we were looking for. Let $F\in[\kappa]\sp{<\omega}$, we must show that there is $m\in\omega$ such that $B\setminus\{x_n:n<m\}\subset\bigcap\{U_\alpha:\alpha\in F\}$. Let $p\in G\cap E_F$ and let $m=\max s_p+1$, we claim that this $m$ works. Let any $k\in\omega$ such that $k\geq m$ and $x_k\in B$. Then there is $q\in G$ such that $k\in s_q$. We may assume that $q\leq p$ so $k\in s_q\setminus s_p$. By property $(e)$, we obtain that $k\in A_\alpha$ for all $\alpha\in F$. Thus, $x_k\in\bigcap\{U_\alpha:\alpha\in F\}$. This is what we needed to finish the proof.
\end{proof}

Given a space $X$ and $A\subset X$ a closed subset, we will say that $S\subset X$ converges to $A$ if for every open set $U$ with $A\subset U$, the set $S\setminus U$ is finite. In this way, the conclusion of Lemma \ref{sequentially} can be stated by saying that the set $B$ converges to $T$.

The proof of the main Theorem \ref{main} can be basically done in two cases. We will first show how to solve these two specific cases in Propositions \ref{casecountable} and Theorem \ref{caseuncountable} and after this we will give the proof of Theorem \ref{main} as a corollary.

\begin{propo}\label{casecountable}
Assume that $\Pint>\omega_1$. Let $K$ be a compact space, $\infty\in K$ such that $K\setminus\{\infty\}$ is countably tight and $A\subset K\setminus\{\infty\}$ a countable dense subset of $K$. Then there is a discrete set $D\subset A$ such that $\infty\in\cl[K]{D}$.
\end{propo}
\begin{proof}
Since $A$ is countable, it is possible to find a collection $\{K_i:i<\omega\}$ of regular closed subsets of $K$ such that $K_i\subsetneq\inte[K]{K_{i+1}}$, $\infty\notin K_i$ for $i<\omega$ and $A\subset\bigcup\{K_i:i<\omega\}$. For each $i<\omega$, let $A_i=\{a(i,n):n<\omega\}$ be an enumeration of $A\cap (K_{i+1}\setminus K_i)$ (where $K_{-1}=\emptyset$). Let $Y=K\setminus\bigcup\{K_i:i<\omega\}$, notice that $Y=K\setminus\bigcup\{\inte[K]{K_i}:i<\omega\}$ so in fact $Y$ is closed and thus compact.

If $\infty$ is isolated in $Y$, let $U$ be an open subset of $K$ with $U\cap Y=\{\infty\}$. Let us assume without loss of generality that $U\cap A_i\neq\emptyset$ for each $i<\omega$. Then pick $p_i\in U\cap A_i$ for each $i<\omega$. For each $i<\omega$, $p_i\in\inte[K]{K_{i+1}}$ and $p_j\in\inte[K]{K_{i+1}}$ if $i<j$. From this, it follows that that $D=\{p_i:i<\omega\}$ is a discrete subset of $A$. Moreover, $\infty\in\cl[K]{D}$ and we are done. 

So assume that $\infty$ is not isolated in $Y$. Now, if there is a countable discrete subset of $Y$ that clusters at $\infty$, we would be done by Lemma \ref{needfreeseq}. So assume that this is not the case, so by Lemma \ref{freeseq}, there is a free $\omega_1$-sequence $\{x_\alpha:\alpha<\omega_1\}\subset Y$ whose only complete accumulation point is $\infty$.

For each $\beta<\omega_1$, consider two disjoint zero sets $Z_\beta$ and $W_\beta$ such that $\{x_\alpha:\alpha<\beta\}\subset Z_\beta$ and $\{x_\alpha:\beta\leq\alpha<\omega_1\}\subset W_\beta$. We may assume that $Z_\beta\cup W_\beta\subset Y$ for all $\beta<\omega_1$ since $Y$ is of type $G_\delta$. For each $\beta<\omega_1$, let $Y_\beta=(\bigcap\{Z_\alpha:\beta<\alpha\})\cap(\bigcap\{W_\alpha:\alpha\leq\beta\})$, notice that $x_\beta\in Y_\beta\subset Y$ and $Y_\beta$ has character $\omega_1$.

Given $\alpha<\omega_1$, $A$ clusters at $x_\alpha$. By Lemma \ref{sequentially}, it is possible to find $E_\alpha\in[\omega]\sp{\omega}$ and $f_\alpha\in{}\sp{E_\alpha}\omega$ such that $\{a(i,f_\alpha(i)):i\in E_\alpha\}$ converges to $Y_\alpha$.

Consider the countable poset $\poset={}\sp{<\omega}\omega$. For each $\alpha<\omega_1$ and $n<\omega$, let $D(\alpha,n)$ be the set of all $p\in\poset$ such that there is $m<\omega$ such that $n\leq m$, $m\in{dom}(p)\cap E_\alpha$ and $p(m)=f_\alpha(m)$. Clearly, $D(\alpha,n)$ is dense in $\poset$ for all $\alpha<\omega_1$. Since $\covM>\omega_1$ there is a generic filter intersecting all these dense sets. Thus, it is possible to define $f\in \baire$ such that $F_\alpha=\{m<\omega:m\in E_\alpha,f(m)=f_\alpha(m)\}$ is infinite.

So let $D=\{a(i,f(i)):i<\omega\}$. It follows that $D$ is a discrete subset of $A$. Given $\alpha<\omega$, let $y_\alpha$ be a cluster point of $\{a(i,f(i)):i\in F_\alpha\}$, notice that $y_\alpha\in Y_\alpha$. Then for each $\beta<\omega_1$, $\{y_\alpha:\alpha<\beta\}\subset Z_\beta$ and $\{y_\alpha:\beta\leq\alpha\}\subset W_\beta$. Thus, $\{y_\alpha:\alpha<\omega_1\}$ is a free $\omega_1$-sequence. Since $K\setminus\{\infty\}$ has countable tightness and $K$ is compact, $\{y_\alpha:\alpha<\omega_1\}$ has $\infty$ as its unique complete accumulation point. Since $y_\alpha\in\cl[K]{D}$ for all $\alpha<\omega_1$, we conclude that $\infty\in\cl[K]{D}$, which is what we wanted to prove.
\end{proof}

The proof of the next result uses Todorcevic's side-condition method for proper forcing. See \cite{todorcevic-notes_forcing}, in particular Chapter 7, for an introduction on this. We will use the usual notation, where $H(\kappa)$ is the set of all sets with transitive closure of cardinality $<\kappa$.

\begin{thm}\label{caseuncountable}
Assume $\PFA$. Let $K$ be a compact space, $\infty\in K$ and $A\in [K]\sp{\omega_1}$ such that $K\setminus\{\infty\}$ has countable tightness, $\chi(\infty,K)=\omega_1$, $A$ is dense in $K$ and no countable subset of $A$ has $\infty$ in its closure. Then there is a discrete set $E\subset A$ such that $\infty\in\cl[K]{E}$.
\end{thm}
\begin{proof}
In \cite{todorcevic-freeseq} and \cite{shapirovskii} (see also \cite[Theorem 3.1]{todorcevic-chaincond_methods}) it is shown that every compact space of countable tightness admits a point-countable $\pi$-base. By considering a maximal pairwise disjoint family of open sets missing $\infty$, we obtain that there is a point-countable $\pi$-base $\B$ of $K$ that moreover consists of open sets with their closures missing $p$. Since $K$ is not separable, $K$ has uncountable $\pi$-weight so $|\B|>\omega$. By refining $\B$, we may also further assume that $\{\cl[K]{B}:B\in\B\}$ is point-countable.

Let $\kappa$ be large enough so that $K,\B\subset H(\kappa)$. Let $\poset$ be the set of all $p=\pair{H_p,\en_p}$ such that the following holds:
\begin{itemize}
\item[(i)] $H_p$ is a finite set of pairs $\pair{a,B}$ where $a\in A\cap B$ and $B\in\B$,
\item[(ii)] if $\pair{a_0,B_0}\neq\pair{a_1,B_1}$ are in $H_p$ then $a_i\notin B_{1-i}$ for $i\in 2$,
\item[(iii)] $\en_p$ is a finite $\in$-chain of countable elementary submodels of $(H(\kappa),\in)$,
\item[(iv)] if $\pair{a_0,B_0}\neq\pair{a_1,B_1}$ are in $H_p$ then there is $N\in\en_p$  and $j\in 2$ such that $a_i\in N$ iff $B_i\in N$ iff $i=j$,
\item[(v)] if $N\in\en_p$ and $\pair{a,B}\in H_p\setminus N$ then for every $a\sp\prime\in A\cap N$ and every $B\sp\prime\in\B$ with $a\sp\prime\in B\sp\prime$ it follows that $a\notin B\sp\prime$.
\end{itemize}
Further, we will say that $q\leq p$ if $H_p\subset H_q$ and $\en_p\subset\en_q$. Then $\pair{\poset,\leq}$ is a poset. Let us assume for the moment that $\pair{\poset,\leq}$ is proper and show how to proceed.

Let $\{U_\alpha:\alpha<\omega_1\}$ be a local base at $\infty$. Fix $\alpha<\omega_1$ for the moment and let 
$$
D_\alpha=\{p\in\poset:\exists\ \pair{a,B}\in H_p\ (a\in U_\alpha)\}.
$$
Let us prove that $D_\alpha$ is dense in $\poset$.

So let $p\in\poset$, we must find an extension of $p$ on $D_\alpha$. Let $H_p=\{\pair{a_i,B_i}:i<k\}$ for some $k<\omega$. Consider $M\prec H(\kappa)$ such that $p\in M$. Clearly, $a_i,B_i\in M$ for all $i<k$ and $\en_p\subset M$. Let $\mathcal{C}=\{B\in\B:\exists\ a\in M\cap A\ (a\in B)\}$. Then since $\B$ is point-countable, $\mathcal{C}$ is countable. Also, consider $V=U_\alpha\setminus(\bigcup\{\cl[K]{B_i}:i<k\})$, which is an open neighborhood of $\infty$, so $V$ is not separable. Thus, there are uncountably many elements of $\B\setminus\mathcal{C}$ contained in $V$, take one of them $B_k$. Finally, choose $a_k\in B_k\cap A$. Define $q$ such that $H_q=H_p\cup\{\pair{a_k,B_k}\}$ and $\en_q=\en_p\cup\{M\}$. Then it is easy to see that $q\in\poset\cap D_\alpha$ and $q\leq p$. 

Thus, by $\PFA$ we can find a generic filter $G\subset\poset$ that intersects $D_\alpha$ for all $\alpha<\omega_1$. Let
$$
E=\{a:\exists\ p\in G, B\in\B\ (\pair{a,B}\in H_p)\}.
$$

Using property (ii) in the definition of $\poset$, it is easy to see that $E$ is a discrete subset of $A$ that intersects every neighborhood of $\infty$ and thus $\infty\in\cl[K]{E}$, which is what we wanted to construct.

Now let us prove that $\pair{\poset,\leq}$ is indeed proper. So let $\theta$ be a large enough cardinal, we may assume that $\kappa<\theta$, and let $N\prec H(\theta)$ be such that $\poset\in N$. Let $M=N\cap H(\kappa)$, so $M\prec H(\kappa)$. So let $p\in\poset\cap N$, we must find a $(N,\poset)$-generic extension of $p$. Since $p\in M$ as well, then $\bar{p}=\pair{H_p,\en_p\cup\{M\}}\in\poset$. We claim that $\bar{p}$ is $(N,\poset)$-generic.

Let $D\in N$ be a dense set of $\poset$ and let $q\in D$ such that $q\leq \bar{p}$. The idea here is to notice that $q=(q\cap M)\cup(q\setminus M)$ and then use elementarity to replace $q\setminus M$ by something inside $N$ that is still compatible with $q$.

There is $k<\omega$ and $T\subset k$ such that $\en_q\setminus M=\{M_i:i\leq k\}$ where $M=M_0\in\ldots\in M_{k}$, $H_q\setminus M=\{\pair{a_i,B_i}:i\in T\}$ where $\pair{a_i,B_i}\subset M_{i+1}\setminus M_i$ for all $i\in T$. We are not going to bother to give names to $q\cap M$, but notice that $q\cap M=\pair{H_q\cap M,\en_q\cap M}\in M$. 

So let $S$ be the set of all $r\in\poset$ that can replace $q\setminus M$. More specifically, 
\begin{itemize}
\item[(a)] $H_r=\{\pair{x_i\sp{r},C_i\sp{r}}:i\in T\}$,
\item[(b)] $\en_r=\{Z_i\sp{r}:i\leq k\}$,
\item[(c)] $\{x_i\sp{r},C_i\sp{r}\}\subset Z_{i+1}\sp{r}\setminus Z_i\sp{r}$ for $i\in T$,
\item[(d)] $Z_i\sp{r}\in Z_{i+1}\sp{r}$ for $i\leq k$,
\item[(e)] $q\cap M\subset Z_0\sp{r}$, and
\item[(f)] $\bar{r}=\pair{H_{q\cap M}\cup H_{r},\en_{q\cap M}\cup\en_{r}}\in D$.
\end{itemize}
Notice that $S\in N$ and $q\setminus M\in S$ so by elementarity $S\cap N\neq\emptyset$.  However, we would like to choose $r\in S\cap N$ in such a way that $\bar{r}$ is compatible with $q$ and this requires more work.

We will need the following claim. Its proof is not hard and follows roughly the same scheme as the proof above that given $\alpha<\omega_1$, the set $D_\alpha$ is dense. So we will leave the proof to the reader.

\vskip6pt
\noindent{\it Claim:} Given any $\B\sp\prime\in[\B]\sp\omega$, there is $r\in S$ such that $C_i\sp{r}\in\B\setminus\B\sp\prime$ for each $i\in T$.
\vskip6pt

Next, for each $r\in S$, let us define $W_r={}\sp{T}[\bigcup\{\cl[K]{C_i\sp{r}}:i\in T\}]$, this is a closed subset of the finite product ${}\sp{T}K$.

Let $x=\pair{x_i:i\in T}\in{}\sp{T}K$ be a $T$-tuple. For each $i\in T$, by the choice of $\B$, there is a set $\B_i\in[\B]\sp{\leq\omega}$ such that $B\in\B_i$ if and only if $x_i\in \cl[K]{B}$ . Then $\B\sp\prime=\bigcup\{\B_i:i\in T\}$ is a countable set. Now let $r\in S$ be given by the Claim above for this choice of $\B\sp\prime$. Notice that this implies that $x\notin W_r$ for this specific choice of $r$. Notice that this implies that $\bigcap\{W_r:r\in S\}=\emptyset$.

By compactness we have in fact that there is a finite set $F\subset S$ such that $\bigcap\{W_r:r\in F\}=\emptyset$. But $W_r\in M$ for every $r\in M$ so by elementarity such an $F$ can be chosen in $M$. So now consider $a\in{}\sp{T}K$ given by $a(i)=a_i$ for each $i\in T$. Since $a\notin\bigcap\{W_r:r\in F\}$, there must exist $r_0\in F\cap M$ such that $a\notin W_{r_0}$.

Thus, we have obtained an element $r_0\in\poset\cap M$ such that $\bar{r_0}\in D$ (the definition of $\bar{r_0}$ is given in item (f)). Using the fact that $a\notin W_{r_0}$, it is easy to see that 
$$
\pair{H_{q\cap M}\cup H_{r_0}\cup (H_q\setminus M),\en_{q\cap M}\cup\en_{r_0}\cup(\en_q\setminus N)}
$$
is an element of $\poset$ that extends both $q$ and $\bar{r_0}$. This concludes the proof that $\poset$ is proper and the proof of this result.
\end{proof}

\begin{proof}[Proof of Theorem \ref{main}]
Let $X$ be such a locally compact space. Clearly, we may assume that $X$ is not compact and call $\infty$ the point at infinity of the one-point compactification $Y=X\cup\{\infty\}$. Given $A\subset X$ such that $\infty\in\cl[Y]{A}$, we must prove that there is a discrete set $D\subset A$ such that $\infty\in\cl[Y]{D}$. All our arguments will be inside $\cl[Y]{A}$ so we will assume, without loss of generality, that $A$ is dense in $X$.

First, let us consider the case when there is a countable subset of $A$ that clusters at $\infty$. Again, by passing to a subspace, we may assume that $A$ is itself countable. Then by Proposition \ref{casecountable} we obtain that $A$ contains a discrete subset that clusters at $\infty$, so we are done.

So assume that every countable subset of $A$ has its closure disjoint from $\infty$. By Lemmas \ref{needfreeseq} and \ref{freeseq} we may assume that there is a free $\omega_1$-sequence $\{x_\alpha:\alpha<\omega_1\}\subset X$ whose only complete accumulation point is $\infty$.

Now, let us recursively define an increasing sequence $\{A_\alpha:\alpha<\omega_1\}$ of countable subsets of $X$ such that $\infty\in\cl[Y]{\bigcup\{A_\alpha:\alpha<\omega_1\}}$ in the following way. In step $\beta<\omega_1$, the countable set $A_{<\beta}=\bigcup\{A_\alpha:\alpha<\beta\}$ has its closure disjoint from $p$ so there is a first ordinal $\lambda(\beta)<\omega_1$ such that $x_{\lambda(\beta)}\notin\cl{A_{<\beta}}$. By countable tightness, there is $S_\beta\in[A]\sp{\leq\omega}$ such that $x_{\lambda(\beta)}\in\cl{S_\beta}$. Let $A_\beta=A_{<\beta}\cup S_\beta$.

Again, we may assume that $A=\bigcup\{A_\alpha:\alpha<\omega_1\}$ without loss of generality. For $\alpha<\omega_1$, let $U_\alpha=Y\setminus\cl{A_\alpha}$. Notice that by countable tightness $\bigcap\{U_\alpha:\alpha<\omega_1\}=\{\infty\}$. By compactness it follows that $\{U_\alpha:\alpha<\omega_1\}$ is a local base at $\infty$. So we may apply Theorem \ref{caseuncountable} and we obtain a discrete $D\subset A$ such that $\infty\in\cl[Y]{D}$. This finishes the proof.
\end{proof}

\section{Character of far and remote points}\label{sectionchar}

Here we continue the study of the character of remote closed sets in separable metrizable spaces that the second-named author started in \cite{hg-remotediscretelygen}. For a Tychonoff non-compact space $X$, we define $\remcard(X)$ to be the minimal $\kappa$ such that there exists set $F$ closed in $\beta X$ remote from $X$ and with $\chi(F)=\kappa$; if no remote sets exist, $\remcard(X)=\infty$. Similarly, define $\farcard(X)$ to be the minimal $\kappa$ such that there exists set $F$ closed in $\beta X$ far from $X$ and with $\chi(F)=\kappa$; if no far sets exist, $\farcard(X)=\infty$. 

In \cite[Proposition 3.1]{hg-remotediscretelygen}, the first part of the following was proved, the second can be proved in a similar way.

\begin{lemma}\label{monotone}
If $Y$ is dense in $X$, then $\remcard(Y)\leq\remcard(X)$ and $\farcard(Y)\leq\remcard(X)$.
\end{lemma}

We define $\remcard=\remcard(\Q)$. The next result gives the new bound we announced in the introduction.

\begin{propo}\label{farseparable}
Let $X$ be a crowded, separable, Tychonoff space and assume that $F\subset\beta X$ is far from $X$. Then the character of $F$ is at least $\d$.
\end{propo}
\begin{proof}
Let $Q\subset X$ be countable and dense. By Lemma \ref{monotone}, we may assume that $Q=X$ so that $X$ is $0$-dimensional. Consider some closed set $F\subset\beta X$ with character $\kappa<\d$, we will prove that $F$ is not far from $X$. We let $\{U_\alpha:\alpha<\kappa\}$ be a base for $F$. It is easy to construct a pairwise-disjoint family $\{W_n:n<\omega\}$ of clopen sets of $\beta X$ such that $X\subset\bigcup\{W_n:n<\omega\}$. For each $n<\omega$, let $\{x(n,i):i<\omega\}$ be an enumeration of $W_n\cap X$.

For each $\alpha<\kappa$, there is a set $E_\alpha\in[\omega]\sp\omega$ and a function $f_\alpha:E_\alpha\to\omega$ such that $x(n,f_\alpha(n))\in U_\alpha$ for each $n\in E_\alpha$. By Theorem 3.6 in \cite{vd62}, there is $f:\omega\to\omega$ such that for all $\alpha<\kappa$ and $m<\omega$ there is $n\in E_\alpha\setminus m$ such that $f_\alpha(n)<f(n)$.

Define $D=\{x(n,i):n<\omega,i\leq f(n)\}$. Then it follows that $D$ is a closed discrete subset of $X$ such that $D\cap U_\alpha\neq\emptyset$ for each $\alpha<\kappa$. Then $\cl[\beta X]{D}\cap F\neq\emptyset$ so $F$ is not far.
\end{proof}

Summarizing what we have about these cardinals so far, we obtain the following.

\begin{thm}\label{farcardinals}
Given any separable, non-compact, metrizable space $X$,
$$
\d\leq\farcard(X)\leq\remcard(X)\leq\cofM
$$
\end{thm}
\begin{proof}
$\farcard(X)\leq\remcard(X)$ follows from the fact that every remote set is far, $\remcard(X)\leq\cofM$ follows from van Douwen's original argument \cite{vd51} (a detailed proof was given in \cite[Theorem 3.4]{hg-remotediscretelygen}) and $\d\leq\farcard(X)$ follows immediately from Proposition \ref{farseparable}.
\end{proof}

We next prove that it is consistent that $\remcard(\omega\times\cantor)$ is not equal to $\d$. The following lemma is based on the classical proof that after adding a random real, ground model reals are meager (see, for example, Theorem 3.20 in \cite{kunen-handbook}).

\begin{lemma}\label{lemmarandom}
After adding one random real there is a closed nowhere dense subset $N\subset\omega\times\cantor$ such that if $F\subset(\omega\times\cantor)\sp\ast$ is a ground model closed subset, then $\cl[\beta(\omega\times\cantor)]{N}\cap F\neq\emptyset$.
\end{lemma}
\begin{proof}
Let $\model$ be a model of ZFC and assume there exists a random real $\dot{r}$ over $\model$. Denote Lebesgue measure on $\cantor$ by $\mu$. Recall that $\cantor$ is a Boolean group with coordinate-wise addition. 

It is a well-known fact from measure theory that there exists a meager, dense and codense subset of $\cantor$ of full measure (for example, see Theorem 1.6 in \cite{oxtoby}). From this, it is possible to find a collection $\{V_k:k<\omega\}$ of dense open subsets of $\cantor$ in $\model$ such that $V_j\subset V_i$ if $i<j$ and $\mu(V_i)<\frac{1}{2\sp{i+1}}$ for all $i<\omega$. For each $k<\omega$, it is possible to find a partition $V_k=\bigcup{\{[s(k,i)]:i<\omega\}}$ where $s(k,i)\in\cantortree$ for all $i<\omega$. Notice that since this is a partition, the following equality holds
$$
(\ast)\ \mu(V_k)=\sum_{i<\omega}{\mu([s(k,i)])},
$$
for all $k<\omega$. Also, for every $k<\omega$, let $N_k=\cantor\setminus V_k$.

Since translations do not change being nowhere dense, the set 
$$
N=\bigcup\{\{k\}\times(N_k+\dot{r}):k<\omega\}
$$
is a closed and nowhere dense in $\omega\times\cantor$. We will prove that $N$ is the set we are looking for. It is enough to prove that given non-compact clopen subset $U\in\model$ of $\omega\times\cantor$, $U\cap N\neq\emptyset$.

Let $x\in\model$ be an accumulation point of the set $\pi[U]$, where $\pi:\omega\times\cantor\to\cantor$ is the projection. From this it is possible to find an infinite set $E\subset\omega$ in $\model$ and for each $i\in E$, $x_i\in\cantor\cap\model$ such that $x=\lim_{i\to\infty}{x_i}$ and $\pair{i,x_i}\in U$ for all $i\in E$. We will prove the following.\vskip10pt

\noindent{\it Claim. }Given $n<\omega$, there is an infinite set $E_n\subset E$ in $\model$ such that
$$
\mu(\bigcup\{V_n+x+x_j:j\in E_n\})\leq\frac{1}{2\sp{n}}.
$$

Let us show first that we are done if we assume the Claim. Given $n<\omega$, let $Z_n=\bigcup\{V_n+x_j:j\in E_n\}$. Since translation by $x$ is a autohomeomorphism of $\cantor$ that preserves measure, the Claim implies that $\mu(Z_n)<\frac{1}{2\sp{n}}$. Then $Z=\bigcap\{Z_n:n<\omega\}$ is a ground-model set of measure zero so $\dot{r}\notin Z$. Then there is $n<\omega$ such that for all $j\in E_n$, $\dot{r}\notin V_n+x_j$. This implies that $x_j\notin V_n+\dot{r}$, or $x_j\in N_n+\dot{r}$. So let $m>n$ such that $m\in E_n$. Then $\pair{m,x_m}\in U\cap N$, which is what we wanted to prove.

So it only remains to prove the Claim. Fix $n<\omega$, we will choose $E_n=\{m_k:k<\omega\}$ recursively. Assume that we are on step $k$ and an increasing sequence $\{m_i:i<k\}$ has been chosen. First, let $j_k<\omega$ be such that
$$
\sum_{j_k\leq i<\omega}\mu([s(n,i)])<\frac{1}{2\sp{k+1}}\mu(V_n),
$$
this is possible by property $(\ast)$ above. Let $\ell_k=\max\{{dom}(s(n,i)):i<j_k\}$. By convergence in the product topology, there is $m_k\in E$ greater than all $\{m_i:i<k\}$ such that $x(i)=x_{m_k}(i)$ for all $i<\ell_k$. Then, $(x+x_{m_k})(i)=0$ for all $i<\ell_k$. From this it follows that if $i<j_k$ then $[s(n,i)]+x+x_{m_k}=[s(n,j)]$. Thus,
$$
(V_n+x+x_{m_k})\setminus V_n\subset\bigcup\{[s(n,i)]:j_k\leq i<\omega\},
$$
which implies that
$$
\mu((V_n+x+x_{m_k})\setminus V_n) <\frac{1}{2\sp{k+1}}\mu(V_n).
$$
This completes the construction. Notice that then we obtain the following
$$
\mu(\bigcup\{(V_n+x+x_{m_k})\setminus V_n:k<\omega\}\leq\sum_{k=0}\sp\infty\frac{1}{2\sp{k+1}}\mu(V_n)=\mu(V_n)
$$
so
$$
\mu(\bigcup\{V_n+x+x_{m_k}:k<\omega\}\leq 2\mu(V_n)=\frac{1}{2\sp{n}},
$$
which implies the statement of the Claim. Thus, this finishes the proof.

\end{proof}

\begin{thm}\label{notcofmeager}
In the model obtained by adding $\omega_2$ random reals to a model of CH we have that $\omega_1=\d<\remcard(\omega\times\cantor)=\cofM=\omega_2=\c$.
\end{thm}
\begin{proof}
Assume CH in the ground model $\model$. Let $\random(\kappa)$ be the poset that adds $\kappa$ random reals (\cite[p. 99]{bart}) and force with $\random(\omega_2)$. In is known (\cite[Model 7.6.8, p. 393]{bart}) that in the resulting model, $\omega_1=\d<\cofM=\omega_2=\c$. So we just have to argue that if $F\subset(\omega\times\cantor)\sp\ast$ has character $\leq\omega_1$ (in the extension), then it is not remote from $\omega\times\cantor$. 

Let $\{U_\alpha:\alpha<\omega_1\}$ be a local base at $F$. Given $\alpha<\omega_1$, $U_\alpha$ is defined by the countable sequence of open sets $\{U_\alpha\cap (\{n\}\times\cantor):n<\omega\}$. So it is known (\cite[Lemma 1.5.7]{bart}) that there is $\beta(\alpha)<\alpha_2$ such that $U_\alpha\in\model\sp{\random(\beta(\alpha))}$. Define $\beta=\sup\{\beta(\alpha):\alpha<\omega_1\}$, this is an ordinal $<\omega_2$.

Then $F$ and its base are in $\model\sp{\random(\beta)}$. But any random real with support in $\random(\omega_2\setminus\beta)$ is random over $\model\sp{\random(\beta)}$. By Lemma \ref{lemmarandom} we obtain that there is a nowhere dense subset $N\subset\omega\times\cantor$ in $\model\sp{\random(\omega_2)}$ such that $\cl[\beta(\omega\times\cantor)]{N}\cap F\neq\emptyset$. So $F$ is not remote, which is what we wanted to prove.
\end{proof}

Thus, we have the following natural questions.

\begin{ques}\label{remotecardquestion}
Is $\remcard=\cofM$ in ZFC?
\end{ques}

\begin{ques}
Let $X$ be a separable metrizable space that is not compact. Is $\remcard=\remcard(X)$?
\end{ques}

Next, we continue with the value of $\farcard$. In this case we will prove that it is equal to $\d$, see Corollary \ref{farchar}. By Proposition \ref{farseparable}, we need to construct a closed subset of $\beta\Q$ far from $\Q$ with character $\d$. We will achieve this by constructing such a set in an ordinal space (Proposition \ref{farinordinal}). First, for the sake of completeness, let us prove in Lemma \ref{smallordinals} that this is not possible in small ordinals.

Let us recall some ordinals obtained by ordinal exponentiation. If $\alpha,\beta$ are ordinals, then $\alpha\cdot\beta$ is the ordinal that is order-isomorphic to $\alpha\times\beta$ with the lexicographic ordering. Whenever $\alpha$ is a limit ordinal, $\alpha\cdot n$ is homeomorphic to $n$ copies of $\alpha+1$ for $n<\omega$ and $\alpha\cdot\omega$ is homeomorphic to $(\alpha+1)\times\omega$. Let $\omega\sp{\cdot 1}=\omega$. For $1\leq n<\omega$, let $\omega\sp{\cdot(n+1)}=\omega\sp{\cdot n}\cdot\omega$. Finally, $\omega\sp{\cdot\omega}=\sup\{\omega\sp{\cdot n}:1\leq n<\omega\}$.

\begin{lemma}\label{smallordinals}
If $\alpha<\omega\sp{\cdot\omega}$, then $\alpha$ does not have far closed sets.
\end{lemma}
\begin{proof}
Clearly, this statement is trivial if $\alpha<\omega$. If $\omega\sp{\cdot n}\leq\alpha<\omega\sp{\cdot(n+1)}$ for some $1\leq n<\omega$, then $\alpha$ is homeomorphic to a subset of a sum of finitely many copies of $\omega\sp{\cdot n}$. So it is enough to prove that $\omega\sp{\cdot n}$ has no far closed sets for any $1\leq n<\omega$, we will prove this by induction.

Clearly, $\omega$ has no far closed sets. So now, let $F$ be a closed subset in the \v Cech-Stone remainder of $\omega\sp{\cdot (n+1)}$, where $n<\omega$. We may think that $F$ is an filter of clopen sets by Stone duality.

Let $T_0$ be the interval $[0,\omega\sp{\cdot n}]$ and for each $k<\omega$, let $T_{k+1}$ be the interval $(\omega\sp{\cdot n}\cdot k,\omega\sp{\cdot n}\cdot(k+1)]$. Then $\omega\sp{\cdot(n+1)}=\bigcup\{T_k:k<\omega\}$ is a partition of $\omega\sp{\cdot(n+1)}$ into clopen sets homeomorphic to $\omega\sp{\cdot n}$.

Now consider the closed discrete set $D=\{\omega\sp{\cdot n}\cdot k:k<\omega\}$ and assume that $F$ does not intersect the closure of $D$. Then there is $U\in F$ such that $D\cap U=\emptyset$. For each $k<\omega$, $U\cap T_k$ is bounded in $T_k$ so it is homeomorphic to some ordinal strictly less than $\omega\sp{\cdot n}$. Thus, $U$ is the free topological sum of $\omega$ ordinals, each less than $\omega\sp{\cdot n}$. This implies that $U$ is isomorphic to an ordinal less than or equal to $\omega\sp{\cdot n}$. By induction, we know that $F$ is in the closure of some closed discrete subset of $U$. This finishes the proof.
\end{proof}

For any space $X$, we recursively define the derived set $X\sp{(\alpha)}$ by recursion in the ordinal $\alpha$. Let $X\sp{(0)}=X$. If $\alpha$ is any ordinal and $X\sp{(\alpha)}$ is defined, let $X\sp{(\alpha+1)}$ be the set of non-isolated points of $X\sp{(\alpha)}$. If $\beta$ is a limit ordinal, let $X\sp{(\beta)}=\bigcap\{X\sp{(\alpha)}:\alpha<\beta\}$. For every space $X$ there is a minimal ordinal $\alpha<\omega$ where $X\sp{(\alpha)}=X\sp{(\alpha+1)}$, this ordinal is called the scattered height of $X$. It is easy to see that for any $\alpha\leq\omega$, $\omega\sp{\cdot\alpha}$ has scattered height $\alpha$.

\begin{propo}\label{farinordinal}
$\omega\sp{\cdot\omega}$ has a far closed subset of character $\d$.
\end{propo}
\begin{proof}
Let $X=\omega\sp{\cdot\omega}$. We will not use the actual identity of the space $X$ but only the following facts. Our space can be written as a topological sum $X=\bigcup\{T_k:k<\omega\}$ where $T_k$ is a compact, linearly ordered space homeomorphic to $\omega\sp{\cdot k}+1$ for every $k<\omega$. Given $k<\omega$, notice that the scattered height of $X_k$ is $k+1$. The following fact will be crucial for our proof.

\vskip10pt
\noindent$(\ast)$ Given $k,n<\omega$, if an open subset $U\subset X_k$ has scattered height at least $n+1$, then there is an interval $[x,y]\in U$ such that $y\in U\sp{(n+1)}$ and an increasing sequence $\{y_i:i<\omega\}\subset(x,y)$ that converges to $y$ with $y_i\in U\sp{(n)}$ for all $i<\omega$.
\vskip10pt

Observe that if $D$ is a closed discrete subset of $X$, then for every $k<\omega$, $D\cap T_k$ is finite. Also, $T_k$ is a countable set for every $k<\omega$. Thus, it follows that there exists a collection of closed discrete sets $\{D_\alpha:\alpha<\d\}$  such that that every time $D\subset X$ is closed and discrete, there exists $\alpha<\d$ such that $D\subset D_\alpha$. Let us assume, without loss of generality, that $\max{(T_k)}\in D_\alpha$ for every $k<\omega$ and $\alpha<\d$.

By recursion, we will construct a collection $\{J_\alpha:\alpha<\d\}$ of clopen sets with the following two properties:
\begin{itemize}
\item[(a)] for every $\alpha<\d$, $D_\alpha\cap J_\alpha=\emptyset$, and
\item[(b)] for every $r<\omega$, $\{\alpha(i):i<r\}\subset\d$ and $n<\omega$, there exists $m<\omega$ such that if $m\leq k<\omega$, then $T_k\cap\big(\bigcap\{J_{\alpha(i)}:i<r\}\big)$ has scattered height at least $n$.
\end{itemize}

Clearly, properties (a) and (b) imply that $\bigcap\{\cl[\beta X]{J_\alpha}:\alpha<\d\}$ is a far closed set of character $\d$. So we just need to carry out this construction to finish the proof.

So assume we are in step $\beta<\d$ of the construction and we have constructed $\{J_\alpha:\alpha<\beta\}$ satisfying (a) and (b). For each $F\in[\beta]\sp{<\omega}$, let $J_F=\bigcap\{J_\alpha:\alpha\in F\}$ and choose an enumeration $\{F_\alpha:\alpha<\beta\}=[\beta]\sp{<\omega}$.

Given $k<\omega$, $D_\beta$ naturally partitions the linearly ordered space $T_k$ into a finite family of pairwise disjoint intervals in the following way. Let $d_\pair{k,0}<d_\pair{k,1}<\ldots<d_\pair{k,r(k)}$ be an increasing enumeration of $D_\beta\cap T_k$. Let also $a_\pair{k,0}=\min{(T_k)}$ and let $a_\pair{k,i+1}$ be the immediate successor of $d_\pair{k,i}$ for every $i<r(k)$. Then we can write $T_k=\bigcup\{[a_\pair{k,i},d_\pair{k,i}]:i\leq r(k)\}$. So consider the set 
$$
S=\{\pair{k,i}:k<\omega,i\leq r(k),[a_\pair{k,i},d_\pair{k,i})\neq\emptyset\}
$$
which gives us a countable infinite way to enumerate all non-empty intervals considered. 

Now, given $s\in S$, we will choose $\{b_s(j):j<\omega\}\subset [a_s,d_s)$. If $d_s$ is a limit, then choose $\{b_s(j):j<\omega\}$ to be increasing and converging to $d_s$. If $d_s$ is isolated, choose $b_s(j)$ to be the immediate predecessor of $d_s$ for all $j<\omega$.

Now, fix $\alpha<\beta$ for a moment, we shall define a function $f_\alpha:S\to\omega$. By (b) in the inductive hypothesis, we may choose a strictly increasing sequence $\{m(i):i<\omega\}\subset\omega$ such that if $m(i)\leq j$, then $T_j\cap J_{F_\alpha}$ has scattered height at least $i+1<\omega$. 

So let $k<\omega$. If $k<m(0)$, define $f_\alpha(\pair{k,i})$ arbitrarily for every $\pair{k,i}\in S$. Now assume that $m(i)\leq k<m(i+1)$. By property $(\ast)$ above, there is an interval $[x,y]\subset T_k\cap J_{F_\alpha}$ with $y\in (T_k\cap J_{F_\alpha})\sp{(i+1)}$. Notice that $y$ is a limit ordinal so $y\neq a_\pair{k,i}$ for any $i<r(k)$. So there is $j<\omega$ such that $y\in (a_\pair{k,j},d_\pair{k,j}]$. Again, since $y$ is a limit ordinal, $\pair{k,j}\in S$. If $y\neq d_\pair{k,j}$, let $f_\alpha(\pair{k,j})$ be large enough so that $y\in [a_\pair{k,j},b_\pair{k,j}(f_\alpha(\pair{k,j}))]$. If $y=d_\pair{k,j}$, then let $f_\alpha(\pair{k,j})$ be large enough so that $[a_\pair{k,j},b_\pair{k,j}(f_\alpha(\pair{k,j}))]\cap(T_k\cap J_{F_\alpha})\sp{(i)}$, this is possible by $(\ast)$. If $i\in\omega\setminus\{j\}$, define $f_\alpha(\pair{k,i})$ arbitrarily. Notice that with this construction, the following property holds.

\vskip10pt
\noindent $(\star)$ For $i<\omega$, $m(i)\leq k$ there exists $j<\omega$ such that 
$$
[a_\pair{k,j},b_\pair{k,j}(f_\alpha(\pair{k,j}))]\cap T_k\cap J_{F_\alpha}
$$
has scattered height at least $i$.
\vskip10pt

Thus, we have defined a collection $\{f_\alpha:\alpha<\beta\}$ of functions between two countable sets. Since $\beta<\d$, there exists a function $f:S\to\omega$ such that for every $\alpha<\beta$, the set $\{s\in S:f_\alpha(s)\leq f(s)\}$ is cofinite in $S$. Finally, we define
$$
J_\beta=\bigcup\{[a_s,b_s(f(s))]:s\in S\},
$$
which is a clopen set that misses $D_\beta$. Property (b) for $\{\alpha(i):i<r\}\subset\beta$ can be easily deduced from property $(\star)$ in the construction. This completes the proof.
\end{proof}

\begin{thm}\label{farseparablefinal}
Let $X$ be any non-compact, separable, metrizable space. Then the following are equivalent.
\begin{itemize}
\item[(a)] $X$ has a far closed set of character $\d$,
\item[(b)] $X$ has a far closed set, and
\item[(c)] either $X$ has a crowded non-compact subset or $X$ has a clopen subspace homeomorphic to $\omega\sp{\cdot\omega}$. 
\end{itemize}
\end{thm}
\begin{proof}
Clearly, (a) implies (b). The fact that (b) implies (c) follows from Lemma \ref{smallordinals}. Now assume (c) and let $X$ be any non-compact, separable, metrizable space. If $X$ has a crowded subset with non-compact closure, then there is a discrete collection $\{F_n:n<\omega\}$ of closed crowded subsets of $X$. Then for each $n<\omega$ it is easy to construct $T_n\subset F_n$ homeomorphic to $\omega\sp{\cdot n}$. By Proposition \ref{farinordinal}, $T=\bigcup\{T_n:n<\omega\}$ has a far closed set $F$ of character $\d$ (in $\beta T$). Since $X$ is normal and $T$ is locally compact, we may assume that $\beta T\setminus T\subset\beta X\setminus X$ so $F\subset\beta X\setminus X$. Also, by normality, $F$ is also far from $X$. Since $T$ is a $G_\delta$ in $X$, it is easy to see that $F$ has character $\d$ in $\beta X$. If we have the case that $X$ has a clopen subspace homeomorphic to $\omega\sp{\cdot\omega}$, it also follows from Proposition \ref{farinordinal} that $X$ has a far closed set of character $\d$. Thus, (a) holds.
\end{proof}

\begin{coro}\label{farchar}
For every non-compact, crowded, separable metrizable space $X$, $\farcard(X)=\d$.
\end{coro}

\section{Far points in countable spaces}\label{sectioncountablefar}

As announced in the introduction, in this section we will study the existence of far points for some countable crowded spaces. Let us explain our method to construct far points, which was inspired by van Douwen's original method from \cite{vd51}.

Consider a normal space $X$ and a collection of closed subsets $\G$ of $X$. We say that $\G$ is \emph{far} if every time $D\subset X$ is closed and discrete there is $G\in\G$ such that $D\cap G=\emptyset$. Also, given $n<\omega$, we will say that $\G$ is $n$-linked if the intersection of any $n$ elements of $\G$ is non-empty. 

Now, assume that $X$ is a normal space that is equal to the disjoint union $\bigoplus_{n<\omega} X_n$ and for each $n<\omega$, $\G_n$ is a $(n+1)$-linked remote collection of closed subsets of $X_n$. Then the collection
$$
\G=\Big\{\bigcup{\{G_n: n<\omega\}}:G_n\in\G_n\textrm{ whenever }n<\omega\Big\}
$$
is a far collection of closed sets. It easily follows that any point in $\bigcap\{\cl[\beta X]{G}:G\in\G\}$ is far.

In our discussion, we are dealing with countable regular crowded spaces. So clearly these spaces are normal and can be partitioned as a disjoint union of countable spaces. Thus, we may use the method described above to prove the existence of remote points in countable spaces.

Our first example is in ZFC but we assume some topological conditions.

\begin{propo}\label{nlinkedfarpicharischar}
Let $n<\omega$. Assume that $X$ is a countable regular crowded space such that every point of $X$ has its character equal to its $\pi$-character. Then $X$ has a far $n$-linked collection of clopen sets.
\end{propo}
\begin{proof}
Let $\kappa$ be the weight of $X$. First, it is not hard to see that it is possible to find a base $\{B_\alpha:\alpha<\kappa\}$ of $X$ with the following properties:
\begin{itemize}
\item[(i)] for every $x\in X$, $\{B_\alpha:\alpha<\chi(x,X)\}$ contains a local base at $x$ in which every element is repeated cofinally, and
\item[(ii)] for every $\omega\leq\beta<\kappa$ there is $\gamma<\kappa$ such that $|\gamma|=|\beta|$ and the Boolean algebra generated by $\{B_\alpha:\alpha<\beta\}$ is contained in $\{B_\alpha:\alpha<\gamma\}$.
\end{itemize}
If $\lambda\leq\kappa$, let $\B_\lambda$ denote $\{B_\alpha:\alpha<\lambda\}$.

Now fix some infinite discrete closed set $D\subset X$. Let $\{U(D,x,0):x\in D\}$ be a discrete collection of pairwise disjoint clopen subsets of $X$ such that $x\in U(D,x,0)$ for each $x\in D$. Clearly, we may assume that $U(D,x,0)\in\B_{\chi(x,X)}$. For each $x\in D$ and $m<\omega$ we shall find $\lambda(D,x,m)<\chi(x,X)$ and a clopen set $U(D,x,m)\in\B_{\lambda(D,x,m)}$ in such a way that the following conditions hold.
\begin{itemize}
\item[(a)] $\lambda(D,x,m)<\lambda(D,x,m+1)$ for all $m<\omega$,
\item[(b)] $U(D,x,m+1)\subsetneq U(D,x,m)$ for all $m<\omega$,
\item[(c)] $U(D,x,0)\in B_{\lambda(D,x,0)}$,
\item[(d)] the Boolean algebra generated by $\B_{\lambda(D,x,m)}\cup\{U(D,x,m+1)\}$ is contained in $\B_{\lambda(D,n,m+1)}$ for all $m<\omega$, and
\item[(e)] if $V\in\B_{\lambda(D,x,m)}$ is such that $V\cap U(D,x,m)\neq\emptyset$, then $(V\cap U(D,x,m))\setminus U(D,x,m+1)\neq\emptyset$.
\end{itemize}

For $m=0$, let $\lambda(D,x,0)$ be any ordinal $<\chi(x,X)$ such that (c) holds. Now assume that we have constructed $U(D,x,m)$ and $\lambda(D,x,m)$ for $m\leq k$ and we want to construct those sets for $m=k+1$. First, notice that the collection $\V=\{B\cap U(D,x,k):B\in\B_{\lambda(D,x,k)}\}\setminus\{\emptyset\}$ is of size $<\chi(x,X)$. Let $W$ be an open set that contains $x$ and witnesses that $\V$ is not a $\pi$-base at $x$. That is, $V\setminus W\neq\emptyset$ for all $V\in\V$. By (i), we may choose $U(D,x,m+1)\in\B_{\chi(x,X)}$ be such that $x\in U(D,x,m+1)\subset W$. Finally, by (ii) we may choose an ordinal $\lambda(D,x,m+1)<\chi(x,X)$ such that the Boolean algebra generated by $\B_{\lambda(D,x,m)}\cup\{U(D,x,m+1)\}$ is contained in $\B_{\lambda(D,x,m+1)}$. Clearly, conditions from (a) to (e) hold and we have finished the construction of these sets. 

Then, given an infinite closed discrete set $D$, we define 
$$
F(D,m)=X\setminus(\bigcup\{U(D,x,m+1):x\in D\}),
$$
which is clearly a clopen subset of $X$ that misses $D$. Also notice that $F(D,m)\subset F(D,m+1)$ for all $m<\omega$. We will next prove that the collection of all sets of the form $F(D,n)$ where $D$ is infinite, closed and discrete is $(n+1)$-linked. Clearly, this is enough to finish the proof.

So let $\D$ be a collection of $n+1$ infinite closed discrete subsets of $X$. Choose $d_0\in D_0\in\D$ be such that $\lambda(D_0,d_0,1)\leq\lambda(D,x,1)$ every time that $D\in\D$ and $x\in D$. Let $\lambda(0)=\lambda(D_0,d_0,1)$. Thus, $W_0=U(D_0,d_0,0)\setminus U(D_0,d_0,1)$ is a non-empty element of $B_{\lambda(0)}$ (properties (d) and (e)) such that $W_0\subset F(D_0,0)$.

If for all $D\in\D\setminus\{D_0\}$ and $x\in D$ we have that $W_0\cap U(D,x,1)=\emptyset$, we would have that $W_0$ is a non-empty subset of $\bigcap\{F(D,n):D\in\D\}$. So assume this is not the case. Then there are $d_1\in D_1\in\D\setminus\{D_0\}$ be such that $W_0\cap U(D_1,d_1,1)\neq\emptyset$ and $\lambda(D_1,d_1,2)\leq\lambda(D,x,2)$ every time that $D\in\D\setminus\{D_0\}$, $x\in D$ and $W_0\cap U(D,x,1)\neq\emptyset$. Notice that $W_0\in \B_{\lambda(D_1,d_1,1)}$ by the minimality of $\lambda(0)$. Let $\lambda(1)=\lambda(D_1,d_1,2)$. By properties (d) and (e), it follows that $W_1=(W_0\cap U(D_1,d_1,1))\setminus U(D_1,d_1,2)$ is a non-empty element of $\B_{\lambda(1)}$. Notice that $W_1\subset F(D_0,0)\cap F(D_1,1)$ so in particular $F(D_0,1)\cap F(D_1,1)\neq\emptyset$.

So in general, assume that we have $W_k\subset\bigcap\{F(D_i,i):i\leq k\}$, where $\lambda(i)=\lambda(D_i,d_i,i+1)\leq\lambda(D,x,i+1)$ for all $D\in\D\setminus\{D_j:j<i\}$, $x\in D$ such that $W_i\cap U(D,x,i+1)\neq\emptyset$ for all $i\leq k$. Further, $W_{i+1}\subset W_i$ for all $i<k$. So, if possible, again let $D_{k+1}\in\D\setminus\{D_i:i\leq k\}$ and $d_{k+1}\in D_{k+1}$ such that $\lambda(D_{k+1},d_{k+1},k+2)$ is minimal. By the minimality of $\lambda(k)$, $W_k\in\B_{\lambda(D_{k+1},d_{k+1},k+1)}$. Let $\lambda(k+1)=\lambda(D_{k+1},d_{k+1},k+2)$. By properties (d) and (e), it follows that $W_{k+1}=(W_k\cap U(D_k,d_k,k))\setminus U(D_k,d_k,k+1)$ is a non-empty element of $\B_{\lambda(k)}$. Thus, we obtain that $W_{k+1}\subset\bigcap\{F(D_i,i):i\leq k+1\}$. 

So with this procedure we obtain that $\bigcap\{F(D_i,i):i\leq n\}\neq\emptyset$. Since $F(D_i,i)\subset F(D_i,n)$ for all $i\leq n$, we obtain that $\bigcap\{F(D_i,n):i\leq n\}\neq\emptyset$. Thus, our far collection is $(n+1)$-linked, which was what we had left to prove.
\end{proof}

As a corollary, we obtain the following.

\begin{thm}\label{farpicharischar}
Every countable regular crowded space where every point has its character equal to its $\pi$-character has far points.
\end{thm}

Now we would like remove the requirement that character equals $\pi$-character everywhere. Unfortunately, we only have a partial result depending on the regularity of the character. Let us start with the case when the character has uncountable cofinality. The proof of the following result is very similar to the proof of Proposition \ref{nlinkedfarpicharischar}.

\begin{propo}
Let $\kappa$ be a cardinal of uncountable cofinality and $n<\omega$. Assume that $X$ is a countable regular crowded space such that every point of $X$ has $\pi$-character equal to $\kappa$. Then $X$ has a far $n$-linked collection of clopen sets.
\end{propo}
\begin{proof}
Consider a $\pi$-base $\{B_\alpha:\alpha<\kappa\}$ of $X$ such that every time $\omega\leq\beta<\kappa$ there is $\gamma<\kappa$ such that the Boolean algebra generated by $\{B_\alpha:\alpha<\beta\}$ is contained in $\{B_\alpha:\alpha<\gamma\}$. Let us denote $\B_\beta=\{B_\alpha:\alpha<\beta\}$ for all $\beta<\kappa$.

Now fix some infinite discrete closed set $D\subset X$. Let $\{U(D,x,0):x\in D\}$ be a discrete collection of pairwise disjoint clopen subsets of $X$ such that $x\in U(D,x,0)$ for each $x\in D$. For each $x\in D$ and $m<\omega$ we shall find $\lambda(D,x,m)<\kappa$, $U(D,x,m)$ a clopen neighbordood of $x$ and a countable subset $N(D,x,m)\subset\lambda(D,x,m)$ such that the following conditions hold.
\begin{itemize}
\item[(a)] $\lambda(D,x,m)<\lambda(D,x,m+1)$ for all $m<\omega$,
\item[(b)] $U(D,x,m+1)\subsetneq U(D,x,m)$ for all $m<\omega$,
\item[(c)] $U(D,x,m)\setminus U(D,x,m+1)=\bigcup\{B_\alpha:\alpha\in N(D,x,m)\}$ for all $m<\omega$,
\item[(d)] the Boolean algebra generated by $\B_{\lambda(D,x,m)}\cup\{B_\alpha:\alpha\in N(D,x,m)\}$ is contained in $\B_{\lambda(D,x,m+1)}$ for all $m<\omega$, and
\item[(e)] if $V\in\B_{\lambda(D,x,m)}$ is such that $V\cap U(D,x,m)\neq\emptyset$, then $(V\cap U(D,x,m))\setminus U(D,x,m+1)\neq\emptyset$.
\end{itemize}

Start by choosing $\lambda(D,x,0)<\kappa$ arbitarily. Now assume that we have constructed everything for $m\leq k$ and that conditions (a) to (e) are satisfied so far. The collection of all $\{B_\alpha:\alpha<\lambda(D,x,k)\}$ is not a local $\pi$-base at $x$ so we may choose a clopen set $U(D,x,k+1)\subset U(D,x,k)$ such that (e) holds. Now choose $N(D,x,k)$ such that (c) holds. Finally, $\lambda(D,x,k+1)$ is defined such that (d) holds, and this can be done by the fact that ${cf}(\kappa)>\omega$ and the choice of our $\pi$-base. This proves that the construction is possible.

Now define 
$$
F(D,m)=X\setminus(\bigcup\{U(D,x,m+1):x\in D\}),
$$
for every closed discrete set $D\subset X$. Clearly, the collection of all $F(D,m)$ where $D$ is infinite, closed and discrete and $m<\omega$ is fixed is a remote collection of clopen sets. Next we will see that it is $(n+1)$-linked when $m=n$.

So let $\D$ be a collection of $n+1$ infinite closed discrete subsets of $X$. Choose $d_0\in D_0\in\D$ be such that $\lambda(D_0,d_0,1)\leq\lambda(D,x,1)$ every time that $D\in\D$ and $x\in D$. Let $\lambda(0)=\lambda(D_0,d_0,1)$. Choose $\alpha(0)\in N(D_0,x_0,0)$ arbitrarily and notice that $W_0=B_{\alpha(0)}$ is a non-empty clopen subset of $F(D_0,0)$ in $\B_{\lambda(0)}$.

First consider the case that for all $D\in\D\setminus\{D_0\}$ and $x\in D$ we had that $W_0\cap U(D,x,1)=\emptyset$. Then $W_0$ is a non-empty subset of $\bigcap\{F(D,n):D\in\D\}$. So assume that this is not the case. Then there are $d_1\in D_1\in\D\setminus\{D_0\}$ such that $W_0\cap U(D_1,d_1,1)\neq\emptyset$ and $\lambda(D_1,d_1,2)\leq\lambda(D,x,2)$ every time that $D\in\D\setminus\{D_0\}$, $x\in D$ and $W_0\cap U(D,x,1)\neq\emptyset$.

Notice that $W_0\in \B_{\lambda(D_1,d_1,1)}$ by the minimality of $\lambda(0)$. Let $\lambda(1)=\lambda(D_1,d_1,2)$. By properties (c) and (e), there exists $\alpha(1)\in N(D_1,x_1,1)$ such that $W_1=W_0\cap B_{\alpha(1)}\neq\emptyset$. Notice that $W_1\in\B_{\lambda(1)}$ by property (d). Thus, $W_1\subset F(D_0,0)\cap F(D_1,1)$ so in particular $F(D_0,1)\cap F(D_1,1)\neq\emptyset$.

So in general, assume that we have $W_k\subset\bigcap\{F(D_i,i):i\leq k\}$, where $\lambda(i)=\lambda(D_i,d_i,i+1)\leq\lambda(D,x,i+1)$ for all $D\in\D\setminus\{D_j:j<i\}$, $x\in D$ such that $W_i\cap U(D,x,i+1)\neq\emptyset$ for all $i\leq k$. Further, $W_{i+1}\subset W_i$ for all $i<k$. So, if possible, again let $D_{k+1}\in\D\setminus\{D_i:i\leq k\}$ and $d_{k+1}\in D_{k+1}$ such that $\lambda(D_{k+1},d_{k+1},k+2)$ is minimal. By the minimality of $\lambda(k)$, $W_k\in\B_{\lambda(D_{k+1},d_{k+1},k+1)}$. Let $\lambda(k+1)=\lambda(D_{k+1},d_{k+1},k+2)$. 

By properties (c) and (e), there exists $\alpha(k+1)\in N(D_{k+1},x_{k+1},k+1)$ such that $W_{k+1}=W_k\cap B_{\alpha(k+1)}\neq\emptyset$ and $W_{k+1}\in\B_{\lambda(k+1)}$ by property (d). Thus, we obtain that $W_{k+1}\subset\bigcap\{F(D_i,i):i\leq k+1\}$. 

So with this procedure we obtain that $\bigcap\{F(D_i,i):i\leq n\}\neq\emptyset$. Since $F(D_i,i)\subset F(D_i,n)$ for all $i\leq n$, we obtain that $\bigcap\{F(D_i,n):i\leq n\}\neq\emptyset$. Thus, our far collection is $(n+1)$-linked, which was what we had left to prove.
\end{proof}

\begin{thm}\label{faruncountblecof}
Let $\kappa$ be a cardinal of uncountable cofinality. Then every countable regular crowded space with all points of $\pi$-character equal to $\kappa$ has far points.
\end{thm}

Fortunately, we are able to obtain far points with a little more generality than this. We need the following easy lemma, we will leave the proof to the reader.

\begin{lemma}\label{smallpicharisclosed}
Let $X$ be a countable regular space and let $\kappa>\omega$ be a cardinal with ${cf}(\kappa)>\omega$. Then the set of all points $x\in X$ with $\pi$-character $<\kappa$ is closed in $X$. 
\end{lemma}

\begin{coro}\label{farregular}
If $X$ is a countable regular crowded regular space of $\pi$-weight a cardinal of uncountable cofinality, then $X$ has far points.
\end{coro}
\begin{proof}
Let $\kappa$ be the $\pi$-weight of $X$. For each $\alpha\leq\kappa$, let $X_\alpha$ be the set of all points in $X$ with $\pi$-character $\leq\alpha$. Then $\{X_\alpha:\alpha\leq\kappa\}$ is an increasing collection of subsets of $X$ with union $X$. Since $X$ is countable and $\kappa$ has uncountable cofinality, there is $\beta<\kappa$ such that $X_\beta=X_\alpha\neq X$ for $\beta\leq\alpha<\kappa$. Let $U=X\setminus X_\beta$. Since $\kappa$ has uncountable cofinality, $X_\beta$ is non-empty and closed by Lemma \ref{smallpicharisclosed}. Thus, $U$ is a non-empty open set of $X$ where all its points has $\pi$-character equal to $\kappa$. Apply Theorem \ref{faruncountblecof} to any clopen subset of $U$ and we are done.
\end{proof}

Unfortunately, we do not have a proof of a version of Corollary \ref{farregular} for countable cofinality. However, we do have the following.

\begin{thm}
Assume that $\kappa>\omega$ is a cardinal with ${cf}(\kappa)=\omega$ and $\c=\kappa\sp{+}$. Then every countable regular crowded space with $\pi$-weight $\kappa$ has far points.
\end{thm}
\begin{proof}
Let $X$ be a countable regular crowded space with $\pi{w}(X)=\kappa$. Let $\{\kappa_n:n<\omega\}\subset\kappa$ be a strictly increasing sequence of cardinals such that $\sup{\{\kappa_n:n<\omega\}}=\kappa$. 
\vskip5pt
\noindent{\it Claim:} There is a discrete pairwise disjoint collection of non-empty clopen sets $\{U_n:n<\omega\}$ of $X$ such that for every $p\in U_n$, $\pi\chi(p)\geq\kappa_n\sp+$. 
\vskip5pt

To prove the Claim, consider the set $P=\{x\in X:\pi\chi(x)=\kappa\}$. For each $n<\omega$, let $P_n=\{x\in X:\pi\chi(x)\leq\kappa_n\}$, which is closed by Lemma \ref{smallpicharisclosed} and the fact that $\kappa_n\sp+$ is regular. We have two cases depending on whether $\cl{P}$ is compact or not. 

If $\cl{P}$ is not compact, then there is an infinite discrete set $\{p_n:n<\omega\}\subset\cl{P}$. We construct a partition $\{A_n:n<\omega\}$ of $X$ into clopen sets such that $p_n\in A_n$. Choose an enumeration $X\setminus\{p_n:n<\omega\}=\{y_n:n<\omega\}$. Then inductively assuming that we have chosen $\{A_n:n<m\}$ for some $m<\omega$ in such a way that $\{y_n:n<m\}\subset\bigcup\{A_n:n<m\}$, let $A_m$ be a clopen set such that $p_m\in A_m$ and $y_m\in\bigcup\{A_n:n\leq m\}$. Clearly with this procedure it is possible to construct such an infinite partition. Notice that $A_n\setminus P_n\neq\emptyset$ for each $n<\omega$. So in this case we may let $U_n$ be a non-empty clopen subset of $A_n\setminus P_n$ for each $n<\omega$.

If $\cl{P}$ is compact, then it must be second countable since it is countable. Thus, there is a point $p\in P$ that is isolated in $\cl{P}$ (but clearly not in $X$). Let $A$ be a clopen set of $X$ such that $A\cap P=\{p\}$. Notice that $A\setminus P_n\neq\emptyset$ for all $n<\omega$. Otherwise $A=P_m\cup\{p\}$ for some $m<\omega$ and $p\in \cl{P_m}=P_m$, which is a contradiction. Recursively construct a pairwise disjoint collection $\{A_n:n<\omega\}$ of clopen subsets of $A$ such that $A_n\subset A\setminus P_n$ for all $n<\omega$. Since $\pi\chi(x)=\kappa>\omega$, there is a clopen set $Q$ such that $x\in Q$ and $A_n\setminus Q\neq\emptyset$ for all $n<\omega$. Thus, for this case let $U_n=A_n\setminus Q$ for each $n<\omega$.

Clearly, in both of the possible cases we have constructed the collection from the Claim.

By the Claim, it is possible to choose a collection $\{D_\alpha:\alpha<\kappa\sp+\}$ of infinite closed discrete subsets of $X$ such that
\begin{itemize}
\item[(i)] $D_\alpha\cap U_n$ is infinite for all $\alpha<\kappa\sp+$ and $n<\omega$ and
\item[(ii)] for every closed discrete set $D\subset X$ there is $\alpha<\kappa\sp+$ such that $D\subset D_\alpha$.
\end{itemize}

By recursion we will construct a centered family $\{F_\alpha:\alpha<\kappa\sp+\}$ of clopen subsets of $X$ such that $D_\alpha\cap F_\alpha=\emptyset$. The filter generated by this familly will be far by property (ii). We will assume that for all $\alpha<\kappa\sp+$ we have the following inductive hypothesis:
\begin{quote}
$(\ast)_{\alpha}$ For every $G\in[\alpha+1]\sp{<\omega}$ there is $N_G<\omega$ such that every time $m\geq N_G$ it follows that $U_m\cap\left(\bigcap\{F_\beta:\beta\in G\}\right)\neq\emptyset$.
\end{quote}

Let $F_0$ be any non-empty clopen set such that $F_0\cap D_0=\emptyset$ and $F_0\cap U_n\neq\emptyset$ for all $n<\omega$. Assume that $\alpha<\kappa\sp+$, that we have constructed $\{F_\beta:\beta<\alpha\}$ and that $(\ast)_\beta$ holds for every $\beta<\alpha$.

Let $\alpha=\bigcup\{J_n:n<\omega\}$ be such that $J_n\subset J_{n+1}$ and $|J_n|\leq\kappa_n$ for all $n<\omega$. Let $\B_\alpha\sp{n}$ denote the Boolean algebra of clopen subsets of $U_n$ generated by $\{F_\beta\cap U_n:\beta\in J_n\}$ for each $n<\omega$. Also, let $D_\alpha=\{x_i:i<\omega\}$ be an precise enumeration and let $\{V_i:i<\omega\}$ be a partition of $X$ into clopen sets such that $\{x_i\}=D_\alpha\cap V_i$ whenever $i<\omega$.

Let $i<\omega$ and $n<\omega$ such that $x_i\in U_n$. Then, since $\pi\chi(x_i)\geq\kappa\sp+_n$, there is a clopen set $W_i\subset U_n\cap V_i$ such that $B\setminus W_i\neq\emptyset$ for all $B\in\B_\alpha\sp{n}\setminus\{\emptyset\}$. Then the set $W=\bigcup\{W_i:i<\omega\}$ is clopen and contains $D_\alpha$. Let $F_\alpha=X\setminus W$.

Now we prove that $(\ast)_\alpha$ holds so let $\beta(0),\ldots,\beta(k)\in\alpha+1$. We may assume that $\beta(k)=\alpha$ and $\beta(j)<\alpha$ if $j<k$. Let $G=\{\beta(j):j<k\}$ and let $\gamma<\alpha$ such that $G\subset\gamma+1$. Since $\gamma<\alpha$, there is some witness $N_G$ for $(\ast)_\gamma$. Let $N<\omega$ be such that $N_G\leq N$ and $\beta(j)\in J_N$ for all $j\leq k$. Let $H=\bigcap\{F_{\beta(j)}:j<k\}$. If $N\leq m<\omega$, by $(\ast)_\gamma$ we have that $H\cap U_m\in B_\alpha\sp{m}\setminus\{\emptyset\}$. Now, let $i<\omega$ be such that $x_i\in U_m$, such an $i$ exists by property (i). By the definition of $W_i$, $(H\cap U_m)\setminus W_i\neq\emptyset$. This easily implies that $\bigcap\{F_{\beta(j)}\cap U_m:j\leq k\}\neq\emptyset$. Thus, $N_{G\cup\{\beta(k)\}}=N$ works as a witness for $(\ast)_\alpha$.

This finishes the construction of $\{N_\alpha:\alpha<\kappa\sp+\}$ and the proof.
\end{proof}

From all the previous results, we have the following.

\begin{coro}\label{farforsmall}
If $\c\leq\aleph_{\omega+1}$, then every regular countable crowded space has far points.
\end{coro}

\end{document}